\documentclass[9pt]{amsart}
\usepackage{amsmath,amsfonts,amsthm,amssymb}
\usepackage{verbatim}
\usepackage{eepic,epic}

\newtheorem{theorem}{Theorem}[section]

\newtheorem{proposition}{Proposition}[section]
\numberwithin{equation}{section}

\newcommand{\al}{\alpha}

 \newcommand{\De}{\Delta}
\newcommand{\la}{\lambda}
\newcommand{\g}{\gamma}
\newcommand{\de}{\delta}

\newcommand{\G}{\Gamma}
\newcommand{\Om}{\Omega}

\newcommand{\ph}{\phi}
\newcommand{\phh}{\varphi}

\newcommand{\ka}{\kappa}

\newcommand{\CC}{{\mathbb C}}
\newcommand{\RR}{{\mathbb R}}
\newcommand{\DD}{{\mathbb D}}
\newcommand{\DDk}{{\mathbb D_{(k)}}}

\newcommand{\NN}{{\mathbb N}}

\newcommand{\CMF}{\mathcal{CMF}}
\newcommand{\BF}{\mathcal{BF}}
\newcommand{\SF}{\mathcal{SF}}
\newcommand{\CBF}{\mathcal{CBF}}

\begin{document}

\title[Estimates for a general fractional relaxation equation]
{Estimates for a general fractional relaxation equation and application to an inverse source problem}

\author{Emilia Bazhlekova}
\address{Institute of Mathematics and
Informatics, Bulgarian Academy of Sciences, Acad. G. Bonchev Str.,
Bl. 8, Sofia 1113, Bulgaria}
\email{e.bazhlekova@math.bas.bg}

\date{ \today}

\begin{abstract}
A general fractional relaxation equation is considered 
with a convolutional derivative in time 
introduced by A. Kochubei (Integr. Equ. Oper. Theory 71 (2011), 583--600).
This equation generalizes the single-term, multi-term and distributed-order fractional relaxation equations. The fundamental and the impulse-response solutions are studied in detail. Properties such as analyticity and subordination identities are established and employed in the proof of an upper and a lower bound. The obtained results extend some known properties of the Mittag-Leffler functions.
As an application of the estimates, uniqueness and conditional stability are established for an inverse source problem for the general time-fractional diffusion equation on a bounded domain.

\end{abstract}

\maketitle

\section {Introduction}

A generalization of the Caputo fractional derivative is introduced in \cite{KGeneral} in the form
\begin{equation}\label{DDk}
(\DDk f)(t)=\frac{d}{d t}\int_0^t k(t-\tau)f(\tau)\,d\tau- k(t)f(0),\ \ t>0,
\end{equation}
where $k\in L^1_{loc}(\RR_+)$ is a nonnegative function. 
The Caputo fractional derivative of order $\al\in(0,1)$ is recovered for $k(t)=\frac{t^{-\alpha}}{\Gamma(1-\alpha)}$, where $\G(.)$ is the Euler Gamma function. If $k(t)=\int_0^1\frac{t^{-\alpha}}{\Gamma(1-\alpha)}\mu(\al)\, d\al$ with $\mu$ being a nonnegative weight function, then $\DDk$ is the distributed-order fractional derivative. An important subclass of weight functions is $\mu(\al)=\sum_{j=1}^m c_j\de(\al-\al_j)$, where $0< \al_j<1$, $c_j>0$, and $\de(.)$ is the Dirac delta function.  In this case $\DDk$ is a linear combination with positive coefficients of Caputo fractional derivatives of orders $\al_j$. 

The general $\DD_{(k)}$-relaxation equation 
 provides relaxation patterns different from
 the exponential relaxation to equilibrium. Such patterns can be caused by inhomogeneities, which distort the exponential
relaxation curve. 
In particular, power-law
relaxation is governed by single and multi-term fractional relaxation equations, while logarithmic relaxation patterns are related to time-derivatives continuously distributed over the interval $[0,1]$. In the context of stochastic processes, general relaxation patterns are studied in \cite{Mainardi, MM, Toaldo}. 

In \cite{Trifce} the general time-fractional diffusion equation with the integro-differential operator $\DD_{(k)}$ acting in time is discussed and its relevance for describing a broad class of anomalous nonscalling patterns is pointed out. The Cauchy problem on unbounded space domain for this generalized diffusion equation is studied in \cite{KGeneral}. In \cite{LuchkoGeneral} some uniqueness and existence results, as well as a maximum principle,  are established for the initial-boundary-value problem. The particular cases of multi-term and distributed-order fractional diffusion equations are studied extensively in the last years in e.g. \cite{A, Burrage2012, Luchkomultiterm, Yamultiterm, K, JLSZ, LuchkoCAMWA, Mijena, ITSF} to mention only few of many recent publications.

Consider the $\DD_{(k)}$-relaxation equation 
\begin{equation}\label{rel}
(\DD_{(k)}u)(t)+\la u(t)=f(t),\ \ \ \ t>0;\ \ \ u(0)=a,
\end{equation}
where $\la\in\RR_+, a\in\RR$. The solution of eq. (\ref{rel}) is represented in the form
\begin{equation}\label{sol}
u(t)=a u(t;\la)+\int_0^t v(t-\tau;\la) f(\tau)\, d\tau,
\end{equation} 
where $u(t;\la)$ is the fundamental solution and $v(t;\la)$ is the impulse-response solution. 
In the particular case of single-term Caputo fractional relaxation equation it is known that the solutions are expressed in terms of Mittag-Leffler functions as follows: $u(t;\la)=E_\al(-\la t^\al)$ and $v(t;\la)=t^{\al-1}E_{\al,\al}(-\la t^\al)$.  Essential feature of these functions is that they are infinitely differentiable on $(0,\infty)$ and completely monotone, i.e. 
\begin{equation}\label{cmfd}
(-1)^n f^{(n)}(t)\ge 0, \mbox{\ for\ all\ } n=0,1,..., t> 0.
\end{equation}
For the theory of functions of Mittag-Leffler type see e.g. \cite{ML, KST, Bessel}. For definitions and properties of completely monotone functions and related classes of functions we refer to \cite{CMF}. 

The following property of Mittag-Leffler functions turns out to be useful in the study of single-term time-fractional diffusion equations, based on eigenfunction expansion: for any $\la$ satisfying $\la\ge\la_1>0$ and any $T>0$ there exists a constant $C>0$, which does not depend on $\la$ (but may depend on $\al, T, \la_1$), such that 
\begin{equation}\label{est0}
{C}\le\la\int_0^T t^{\al-1}E_{\al,\al}(-\la t^\al)\,dt<1,\ \ \ 0<\al\le 1.
\end{equation}
The bound from above is useful in the study of regularity of inhomogeneous equations, see  \cite{SaYa}, and the estimate from below in the study of inverse source problems with final overdetermination. It is used e.g. in \cite{SaYaInv, WeiWang2014, ZhangXu2011, ZhangWei2013} for the proof of uniqueness and conditional stability of such inverse problems. 
 
Concerning the general relaxation equation (\ref{rel}), it is established in \cite{KGeneral} that under some additional conditions on the kernel $k(t)$ 
the fundamental solution is continuous on $[0,\infty)$, infinitely differentiable and completely monotone on $(0,\infty)$. 

The main aim of this work is to generalize estimates (\ref{est0}) to the case of the $\DDk$-relaxation equation (\ref{rel}). The proof consists of two main steps: establishing a subordination identity of the form 
\begin{equation}\label{sub}
u(t;\la)=\int_0^\infty \phi(t,\tau)e^{-\la \tau}\, {d}\tau, \ \ t>0,
\end{equation}
where $\phi(t,\tau)$ is a probability density function, 
and the proof of analyticity of $u(t;\la)$ in $t>0$. As usual, the proofs are based on the Laplace transform technique.
As an application of the obtained estimates, uniqueness and a conditional stability result are established for an inverse source problem for the general time-fractional diffusion equation on a bounded domain.

The paper is organized as follows. In Section 2. the fundamental and the impulse-response solutions of the $\DDk$-relaxation equation are studied in detail. The results are applied in Section 3. to study a direct and an inverse source problem for the $\DDk$-diffusion equation on a bounded space domain. Definitions and some useful properties of special classes of functions related to complete monotonicity are given in an Appendix.

\section{$\DDk$-relaxation equation}

Following \cite{KGeneral}, in this work we assume that the Laplace transform $\widehat{k}(s)$ of the kernel $k(t)$ exists for all $s>0$, $\widehat{k}(s)\in\SF$ - the class of Stieltjes functions (for the definition see Appendix) and 
\begin{equation}\label{k}
\widehat{k}(s)\to 0, s\widehat{k}(s)\to \infty \ \mathrm{as} \ s\to \infty;\ \  
\widehat{k}(s)\to \infty, s\widehat{k}(s)\to 0 \ \mathrm{as} \ s\to 0. 
 \end{equation} 
The above assumptions on $\widehat{k}(s)$
imply $k(t)\in\CMF$, see \cite{KGeneral}.

The unique solvability of the $\DDk$- relaxation equation is proved in \cite{KGeneral}. Here we establish some additional representations and properties of the fundamental and impulse-response solutions.

By applying Laplace transform to equation (\ref{rel}), we obtain the following representations 
in Laplace domain 
\begin{equation}\label{Lsol}
\widehat{u}(s;\la)=\frac{g(s)}{s(g(s)+\la)},\ \ \widehat{v}(s;\la)=\frac{1}{g(s)+\la},\ \ g(s):=s\widehat{k}(s).
\end{equation}
The assumptions on the kernel $k(t)$ are equivalent to the following assumptions on the function $g(s)$
\begin{equation}\label{g}
g(s)\in \CBF; \ \ \frac{g(s)}{s}\to 0, g(s)\to \infty \ \mathrm{as} \ s\to \infty;\ \  
\frac{g(s)}{s}\to \infty, g(s)\to 0 \ \mathrm{as} \ s\to 0, 
\end{equation} 
where $\CBF$ denotes the class of complete Bernstein functions, see Appendix.

First we prove subordination relations of the form (\ref{sub}) for $u(t;\la)$ and $v(t;\la)$.

\begin{theorem}
The fundamental solution $u(t;\la)$ and the impulse-response solution $v(t;\la)$ of problem (\ref{rel}) satisfy the subordination identities  
\begin{equation}\label{sub1}
u(t;\la)=\int_0^\infty \phi(t,\tau)e^{-\la \tau}\, {d}\tau, \ \ t>0,
\end{equation} 
\begin{equation}\label{sub2}
v(t;\la)=\int_0^\infty \psi(t,\tau)e^{-\la \tau}\, {d}\tau, \ \ t>0,
\end{equation}
where the functions $\ph(t,\tau)$ and $\psi(t,\tau)$ 
 obey the properties 
\begin{equation}\label{12}
\ph(t,\tau)\ge 0,\ \psi(t,\tau)\ge 0;\ \ \ \int_0^\infty \ph(t,\tau)\, d\tau=1,\ \int_0^\infty \psi(t,\tau)\, dt=1.
\end{equation}
 \end{theorem}
\begin{proof}
Let us define a function $\ph(t,\tau)$ by the identity
\begin{equation}\label{defphiint}
\ph(t,\tau)=\frac{1}{2\pi \mathrm{i}} \int_{\g-\mathrm{i}\infty}^{\g+\mathrm{i}\infty}e^{st-\tau g(s)}\,\frac{g(s)}{s}\,ds, \ \ \g,t,\tau>0;
\end{equation}
Then the Laplace transform of $\ph(t,\tau)$ with respect to $t$ 
is given by
\begin{equation}\label{phi}
\widehat{\ph}(s,\tau)=\int_0^\infty e^{-st}\ph (t,\tau)\, dt=\frac{g(s)}{s}e^{-\tau g(s)},\ \ s,\tau>0,
\end{equation}
If we define a function $u(t;\la)$ by (\ref{sub1}), then for its Laplace transform we obtain
$$
\int _0^\infty e^{-st}u(t;\la)\, dt=\int_0^\infty\widehat{\ph}(s,\tau)e^{-\la\tau}\, d\tau=\frac{g(s)}{s}\int_0^\infty e^{-\tau g(s)}e^{-\la\tau}\, d\tau=\frac{g(s)}{s(g(s)+\la)}.
$$
Comparing this result to (\ref{Lsol}), it follows by the uniqueness of the Laplace transform  that $u(t;\la)$ defined by (\ref{sub1}) and the fundamental solution of (\ref{rel}) coincide. In this way (\ref{sub1}) is established. 

For (\ref{sub2}) we define $\psi(t,\tau)$ by the identity
\begin{equation}\label{psi}
\widehat{\psi}(s,\tau)=\int_0^\infty e^{-st}\psi (t,\tau)\, dt=e^{-\tau g(s)},\ \ s,\tau>0,
\end{equation}
and use an analogous argument.

To prove that $\ph(t,\tau)$ and $\psi(t,\tau)$ are nonnegative, we use a standard argument based on properties (A)-(C) of Proposition~\ref{Appendix} in the Appendix. Since $g(s)\in\CBF$ then $g(s)/s\in\CMF$ and $e^{-\tau g(s)}\in \CMF$. 
Therefore, (\ref{phi}) and (\ref{psi}) imply that $\widehat{\psi}(s,\tau)\in\CMF$
and $\widehat{\ph}(s,\tau)\in\CMF$ as a product of two completely monotone functions and, hence, $\ph(t,\tau)\ge 0$ and $\psi(t,\tau)\ge 0$ by the Bernstein's theorem. 

The integral identity in (\ref{12}) for $\ph(t,\tau)$ can be obtained as a particular case of (\ref{sub1}) taking $\la = 0$  and noticing that the unique solution in this case is $u(t;0)\equiv 1$. 
An alternative proof follows from the definition (\ref{defphiint}) of $\phi(t,\tau)$:
\begin{eqnarray}
\int_0^\infty \phi(t,\tau){d}\tau&=&\frac{1}{2\pi \mathrm{i}} \int_{\g-\mathrm{i}\infty}^{\g+\mathrm{i}\infty}e^{st}\frac{g(s)}{s}\int_0^\infty\exp\left(-\tau g(s)\right)\,{d}\tau{d}s\nonumber\\
&=&\frac{1}{2\pi \mathrm{i}} \int_{\g-\mathrm{i}\infty}^{\g+\mathrm{i}\infty}\frac{e^{st}}{s}{d}s=1.\nonumber
\end{eqnarray}

The integral identity in (\ref{12}) for $\psi(t,\tau)$ follows taking $s\to 0$ in (\ref{psi}).
 \end{proof}

Subordination identities (\ref{sub1}) and (\ref{sub2}) can be also deduced as particular cases of the generalized multiplication theorem of Efros \cite{Ditkin}. This theorem is used e.g. in \cite{MMAS} to obtain integral representation of multivariate Mittag-Leffler functions. 

Let us note that more general subordination results for the case of $\DDk$-heat equation are given in \cite{KGeneral}. 
Stochastic interpretation is discussed in \cite{MM, Toaldo} in the context of stable subordinators. 
In this work we restrict ourselves to the scalar case ($\la$=constant) and consider the subordination identities (\ref{sub1}) and (\ref{sub2}) purely as useful representations of the solutions of the $\DDk$-relaxation equation.

Further, we use the following characterization of functions which are holomorphic in a sector,
see \cite{Pruss}, Theorem 0.1.

\begin{proposition}\label{Pruss}
Let $f$ be a function defined on $(0,\infty)$ and $\theta_0\in (0,\pi/2]$. Then the following are equivalent:\\
(i) $f(s)$ admits holomorphic extension to the sector $|\arg s|<\pi/2+\theta_0$ and $sf(s)$ is bounded on each sector $|\arg s|\le \pi/2+\theta$, $\theta<\theta_0$;\\
(ii) there is a function $v(t)$ holomorphic for $|\arg t|<\theta_0$ and bounded on each sector  $|\arg t|\le\theta<\theta_0$, such that $f(s)=\widehat{v}(s)$ for each $s>0$.
\end{proposition}

In the next theorem, further useful properties of the functions $u(t;\la), v(t;\la)$ are summarized, leading to the central result of this work, estimates (\ref{est}). 

\begin{theorem}\label{T}
For any $\la> 0$ the functions $u(t;\la)$ and $v(t;\la)$ admit holomorphic extensions to the half-plane $\Re \,t>0$ and 
\begin{eqnarray}
&& u(t;\la), v(t;\la)\in\CMF \ {\mbox in} \ t> 0; \label{cm}\\
			&&u(0;\la)=1; \ \  0<u(t;\la)<1, \ \ v(t;\la)>0,\ \ \ t>0;\label{p2}\\
			&&\frac{\partial}{\partial t}u(t;\la)=-\la v(t;\la).\label{p1}
	\end{eqnarray}
	For any $\la\ge\la_1>0$ and $t>0$
	\begin{equation}\label{p3}
			u(t;\la)\le u(t;\la_1), \ \ v(t;\la)\le v(t;\la_1),
	\end{equation}
	and 
	\begin{equation}\label{est}
	{C}\le\la\int_0^T v(t;\la)\,dt<1, \ \ T>0,
	\end{equation}
	where the constant $C=1-u(T;\la_1)>0$ is independent of $\la$.
	\end{theorem}
	\begin{proof}
	First, applying Proposition~\ref{Pruss}, we prove that functions $u(t;\la)$ admits holomorphic extensions to the half-plane $\Re t>0$. Since the function $v(t)=e^{-\la t}$ is holomorphic for all $t\in \CC\backslash[0,\infty)$ and bounded for $\Re t>0$, then, using that (ii) implies (i), it follows that the Laplace transform $\widehat{v}(s)=\frac{1}{s+\la}$ admits holomorphic extension to the sector $|\arg s|<\pi$ and 
	\begin{equation}\label{b1}
	\left|s\widehat{v}(s)\right|=\left|\frac{s}{s+\la}\right|\le M,\ |\arg s|\le \theta, \theta<\pi.
	\end{equation}
	Since $g(s)\in\CBF$ this function admits holomorphic extension to the sector $|\arg s|<\pi$ and therefore, in view of (\ref{Lsol}), this will hold also for $\widehat{u}(s;\la)$. 
	Consider the kernel $\ka(t)$ defined by the identity $\widehat{\ka}(s)\widehat{k}(s)=1/s$. It is proven in \cite{KGeneral} that $\ka(t)\in\CMF$. Since $g(s)=1/\widehat{\ka}(s)$, property (D) of Proposition~\ref{Appendix} implies 
	$$
	|\arg g(s)|\le |\arg s|,\ \ s\in \CC\backslash (-\infty,0],
	$$
	which together with (\ref{b1}) gives  
	$$
	\left|s\widehat{u}(s;\la)\right |=\left|\frac{g(s)}{g(s)+\la}\right |\le M,\ |\arg s|\le \theta, \theta<\pi.
	$$
	From implication (i)$\Rightarrow$ (ii) in Proposition~\ref{Pruss}, it follows that function $u(t;\la)$ admits holomorphic extensions for $\Re t>0$. The analyticity of $v(t;\la)$ is then inferred taking into account relation (\ref{p1}). 
	
	Complete monotonicity of $u(t;\la)$ is proven in \cite{KGeneral}, Theorem~2. Then, relation (\ref{p1}) imply also $v(t;\la)\in\CMF$.
	
	Applying the property of Laplace transform $f(0)=\lim_{s\to\infty} s\widehat {f}(s)$ and (\ref{Lsol}) it follows that $u(0;\la)=1$. Since $u(t;\la), v(t;\la) \in \CMF$, they are nonnegative and nonincreasing functions for $t>0$. In fact, these functions are positive and strictly decreasing. This follows from their analyticity. Indeed, if $u(\tau;\la)=1$ for some $\tau>0$ then $u(t;\la)=1$ for all $t\in[0,\tau]$ and therefore, from the analyticity of $u(t;\la)$ it should be equal to $1$ for all $t\ge 0$, which is a contradiction. Analogous argument shows that $u(t;\la), v(t;\la)\neq 0$ for all $t\ge 0$.
	
	From (\ref{Lsol}) and $u(0;\la)=1$ we deduce
	$$
	\widehat{\frac{\partial u}{\partial t}}(s,\la)=\frac{g(s)}{g(s)+\la}-1=-\frac{\la}{g(s)+\la}=-\la\widehat{v}(s,\la).
	$$
	Identity (\ref{p1}) then follows from the uniqueness property of the Laplace transform. 
	
	The inequalities (\ref{p3}) follow directly from the subordination identities (\ref{sub1}) and (\ref{sub2}).   Indeed, for $\la\ge\la_1$ 
	$$
	u(t;\la)=\int_0^\infty \phi(t,\tau)e^{-\la \tau}\, {d}\tau\le \int_0^\infty \phi(t,\tau)e^{-\la_1 \tau}\, {d}\tau=u(t;\la_1), 
	$$
	and analogously for $v(t;\la)$. Here the nonnegativity property of the function $\phi(t,\tau)$ (resp. $\psi(t,\tau)$) is essential.
	
	Applying (\ref{p1}) we deduce
	$$
	\la\int_0^T v(t;\la)\,dt=1-u(T;\la).
	$$
	This together with the first inequality in (\ref{p3}) and $0<u(T;\la)<1$ implies (\ref{est}).
		\end{proof}

    Let us note that all statements in Theorem~\ref{T} extend known properties of Mittag-Leffler functions. Some of them are also reported for fractional multi-term and distributed-order equations \cite{K,Yamultiterm}. Analyticity of $u(t;\la)$ is established for distributed-order equations with continuous weight function in \cite{LuchkoCAMWA}. The upper bound in (\ref{est}) is given in the multi-term case  
		in \cite{CAA}.

\section{$\DDk$-diffusion equation }

In this section we use the estimates (\ref{est}) to study the general $\DDk$-diffusion equation on a bounded domain. It is obtained from the classical diffusion equation replacing the first order time-derivative by the integro-differential operator $\DDk$ defined in (\ref{DDk}).

Let $\Omega\subset\RR^d$ be a bounded domain with sufficiently smooth boundary $\partial\Omega$, and $T>0$. 
Consider the initial-boundary-value problem 
\begin{equation}\label{p}
  \begin{aligned}
\DD_{(k)}u(x,t)&=\Delta u(x,t)+F(x,t),\ \ x\in \Omega,\ \ t\in(0,T),\\ 
u(x,t)&=0,\ \  x\in\partial\Omega, \ t\in(0,T),\\ 
u(x,0)&=a(x),\ \ x\in\Omega, 
  \end{aligned}
\end{equation}
where the operator $\DD_{(k)}$ acts with respect to time variable and $\De$ is the Laplace operator acting on space variables.


Next we give an application of the estimates (\ref{est}) for obtaining a regularity estimate for the inhomogeneous direct problem (\ref{p}) and for the study of uniqueness and conditional stability of a simple inverse source problem with final overdetermination.

\subsection{Direct problem}

Define the Laplace operator $\De$ in $L^2(\Om)$ with domain $D(\De)= H_0^1(\Omega) \cap H^2(\Omega)$, where $H_0^1(\Omega)$ and $H^2(\Omega)$ denote Sobolev spaces.
Denote by $\{-\la_n,\phh_n\}_{n=1}^\infty$ the corresponding eigensystem. Then $0<\la_1\le\la_2\le...,\ \la_n\to\infty$ as $n\to\infty$, and $ \{\phh_n\}_{n=1}^\infty$ form an orthonormal basis of $L^2(\Om)$.

Applying eigenfunction decomposition, we obtain from (\ref{sol}) the following formal representation of the solution of problem (\ref{p}) 
\begin{equation}\label{FS}
u(x,t)=\sum_{n=1}^\infty a_n u_n(t)\phh_n(x)+\sum_{n=1}^\infty\left(\int_0^t v_n(t-\tau)F_n(\tau)\,d\tau \right)\phh_n(x)
\end{equation}
where $u_n(t)=u(t;\la_n)$, $v_n(t)=v(t;\la_n)$ are the fundamental and impulse-response solution of the relaxation equation~(\ref{rel}) with $\la=\la_n$, $n\in\NN$, and
$$
a_n=(a,\phh_n), \ \ F_n(t)=(F(.,t),\phh_n),\ \ n\in\NN,
$$
with $(.,.)$ denoting the inner product in $L^2(\Om)$.

 Let us consider the 
case $a= 0$ and $F\neq 0$, $F\in L^2(0,T;L^2(\Om))$. 
We use the upper bound in (\ref{est}) 
to prove that the solution $u(x,t)$ satisfies the estimate:
\begin{equation}\label{est1}
\|\De u\|_{L^2(0,T;L^2(\Om))}\le \|F\|_{L^2(0,T;L^2(\Om))}.
\end{equation}
The norm in the space $L^2(0,T;L^2(\Om))$ of a function $f(x,t)$ is given by:
$$\|f\|^2_{L^2(0,T;L^2(\Om))}=\int_0^T \|f\|^2_{L^2(\Om)}\, dt=\int_0^T \sum_{n=1}^\infty |f_n(t)|^2\, dt=\sum_{n=1}^\infty \|f_n(t)\|^2_{L^2(0,T)}.$$
Applying the Young inequality for the convolution and the upper bound in (\ref{est}) 
it follows:
\begin{equation*}
\left\|\int_0^t v_n (t-\tau)F_n(\tau)\,d\tau\right\|_{L^2(0,T)}^2
\le 
\left(\int_0^T v_n(t)\,dt\right)^2 \int_0^T |F_n(t)|^2\,dt
\le \frac{1}{\la_n^2} \int_0^T |F_n(t)|^2\,dt.
\end{equation*}
Inserting this estimate in the spectral decomposition (\ref{FS}) we get (\ref{est1}):
\begin{eqnarray}
\| \De u\|_{L^2(0,T;L^2(\Om))}^2&\le& \sum_{n=1}^{\infty}\la_n^2\left\|\int_0^t v_n(t-\tau)F_n(\tau)\,d\tau\right\|_{L^2(0,T)}^2\nonumber\\
&\le& \sum_{n=1}^{\infty}\int_0^T |F_n(t)|^2\,dt= \|F\|_{L^2(0,T;L^2(\Om))}^2.\nonumber
\end{eqnarray}

\subsection{Inverse problem: uniqueness and a conditional stability estimate }

For the inverse source problem for the $\DDk$-diffusion equation we take $a=0$ and $F(x,t)=f(x)q(t)$ in (\ref{p}), where the function $q\in C[0,T]$ is known and satisfies $q(t)\ge q_0>0$ for all $t\in [0,T]$. The problem is to determine $f(x)$ and $u(x,t)$, 
satisfying (\ref{p}) and the additional condition
\begin{equation}\label{h}
u(x,T)=h(x),\ x\in\overline{\Om}.
\end{equation}
Taking $t=T$ in the formal expansion (\ref{FS})  of the solution of (\ref{p}) we obtain
 \begin{equation}\label{hh}
h(x)=\sum_{n=1}^\infty f_n\left(\int_0^T v_n(T-\tau)q(\tau)\,d\tau \right)\phh_n(x), 
\end{equation}
where $f_n=(f,\phh_n)$. 
Introducing the notations $h_n=(h,\phh_n)$ and $Q_n(t)=\int_0^t v_n(t-\tau)q(\tau)\,d\tau$, (\ref{hh}) gives 
\begin{equation}\label{fh}
h_n=f_n Q_n(T).
\end{equation}
 Since $Q_n(T)\ge q_0\int_0^T v_n(\tau)\,d\tau$ and $Q_n(T)\le \|q\|_{C[0,T]}\int_0^T v_n(\tau)\,d\tau$, the bounds in (\ref{est}) imply
\begin{equation}\label{Q}
0<\underline{C}/\la_n\le Q_n(T)\le \overline{C}/\la_n,
\end{equation}
where the constants $\underline{C}$ and $\overline{C}$ do not depend on $n$.

In particular, $Q_n(T)>0$. This implies that the solution $\{f(x), u(x,t)\}$ of problem (\ref{p}), (\ref{h}) is unique. Indeed, if $h(x)=0$ then $f(x)=0$ by (\ref{fh}) and from the uniqueness of the direct problem, also $u(x,t)=0$.

Estimates (\ref{Q}) for $Q_n(T)$ and (\ref{fh}) show that $h(x)$ has a better regularity than $f(x)$. More precisely,
\begin{equation}\label{reg}
\underline{C}\|f\|_{L^2(\Om)}\le \|h\|_{H^2(\Om)}\le \overline{C}\|f\|_{L^2(\Om)}.
\end{equation}
Therefore, the inverse problem is moderately ill posed.

The lower bound in (\ref{Q}) can be used to prove the following conditional stability result: If $\|f\|_{H^2(\Om)}\le E$ then 
\begin{equation}\label{cs}
\|f\|_{L^2(\Om)}\le \underline{C}^{-1/2} E^{1/2} \|h\|^{1/2}_{L^2(\Om)}.
\end{equation}
Indeed, (\ref{fh}) and the Cauchy-Schwarz inequality yield
\begin{equation}\label{ll}
\|  f\|_{L^2(\Om)}^2= \sum_{n=1}^{\infty}f_n^2=
\sum_{n=1}^{\infty}\frac{h_n}{Q_n^2(T)}h_n
\le\left(\sum_{n=1}^{\infty}\frac{h_n^2}{Q_n^4(T)}\right)^{\frac12}\|  h\|_{L^2(\Om)}.
\end{equation}
Applying (\ref{Q}), the obtained sum is bounded as follows 
$$
\sum_{n=1}^{\infty}\frac{h_n^2}{Q_n^4(T)}=\sum_{n=1}^{\infty}\frac{f_n^2}{Q_n^2(T)}\le \underline{C}^{-2}\sum_{n=1}^{\infty}\la_n^2f_n^2\le \underline{C}^{-2}\|f\|^2_{H^2(\Om)}\le \underline{C}^{-2} E^2.
$$
Plugging this estimate in  (\ref{ll}) completes the proof of (\ref{cs}).

\section{Concluding remarks}

Relaxation and diffusion equations associated with a general convolutional derivative in time are studied. Estimates for the solution of the relaxation equation are established and applied to study the diffusion equation on a bounded space domain via eigenfunction expansion. A regularity estimate for an inhomogeneous direct problem is proven, as well as uniqueness and a conditional stability for an inverse source problem with final overdetermination.

The established estimates can serve as a basis for establishing  more elaborate regularity results, which extend the results in \cite{SaYa} for the single-term fractional diffusion equation. 
Another feasible generalization is to apply the developed method to a diffusion equation with nonlocal boundary conditions, as the one studied in \cite{ITSF0}. 

\section*{Acknowledgements}
This work is performed in the frames of the bilateral research project between Bulgarian and Serbian academies of sciences "Analytical and numerical methods for differential and integral equations and mathematical models of arbitrary (fractional or high integer) order".

\section*{Appendix}
Here we give some notations, definitions and properties of special classes of functions related to completely monotone functions, see (\ref{cmfd}). 


The characterization of the class of completely monotone functions ($\CMF$) is given by the Bernstein's theorem which states that
a function is completely monotone if and only if it can be represented as the Laplace transform of a non-negative measure (non-negative function or generalized function).

The class of Stieltjes functions ($\SF$) consists of all functions defined on $(0,\infty)$ which 
have the representation
$$
\phh(\la)=\frac{a}{\la}+b+\int_0^\infty e^{-\la \tau}\psi(\tau)\,d\tau,
$$
where $a,b\ge 0$, $\psi\in \CMF$ and the Laplace transform of $\psi$ exists for any $\la>0$.
Obviously, $\SF \subset\CMF$. 

A non-negative function $\phh$ on $(0,\infty)$ is said to be a Bernstein function ($\phh\in \BF$) if 
$\phh'(\la)\in \CMF$; $\phh(\la)$ is said to be a complete Bernstein functions ($\CBF$) if and only if $\phh(\la)/\la\in \SF$. 
We have the inclusion $\CBF\subset \BF.$ 

Basic examples of Stieltjes and complete Bernstein functions are the following: if $\al\in[0,1]$ then $\la^{-\al}\in \SF$ and $\la^{\al}\in\CBF
.$

 \begin{proposition}\label{Appendix}
 The following properties are satisfied:\\
{\rm(A)} 
The class $\CMF$ is closed under point-wise addition and multiplication.\\
{\rm(B)} If $\phh\in \BF$ then $\phh(\la)/\la\in \CMF$.\\
{\rm(C)} If $\phh\in \CMF$ and $\psi\in \BF$ then the composite function $\phh(\psi)\in\CMF$.\\
{\rm(D)} If $k\in L^1_{loc}(\RR_+)$ and $k\in\CMF$ then the function $\widehat{k}(\la)$ can be analytically extended to $\CC\backslash (-\infty,0]$ and $$|\arg \widehat{k}(\la)|\le |\arg \la|,\ \la\in \CC\backslash (-\infty,0].$$
\end{proposition}

For more details on these special classes of functions and proofs of the properties we refer to \cite{CMF} and \cite{Pruss}. For the proof of (D) see \cite{Pruss}, Example 2.1.

\bibliographystyle{abbrv}
\bibliography{mybibMMAS}

\end{document}